\documentclass[a4paper,12pt]{article}
\usepackage{a4wide}
\usepackage{amsmath}
\usepackage{amssymb}
\usepackage{amsthm}
\usepackage{latexsym}
\usepackage{graphicx}
\usepackage[english]{babel}
\usepackage{makeidx}

\newtheorem{exm} [subsection]{Example}

\newtheorem{prop}[subsection]{Proposition}

\newtheorem{teor}[subsection]{Theorem}
\newtheorem{lema}[subsection]{Lemma}
\newtheorem{cor} [subsection]{Corollary}
\newcommand{\Zng}{$\mathbb Z^n$-graded $S$-module}

\def\sdepth{\operatorname{sdepth}}
\def\depth{\operatorname{depth}}

\def\reg{\operatorname{reg}}

\def\Ass{\operatorname{Ass}}

\begin{document}
\selectlanguage{english}
\frenchspacing

\large
\begin{center}
\textbf{On the Stanley depth of a special class of Borel type ideals}

Mircea Cimpoea\c s
\end{center}
\normalsize

\begin{abstract}
We give sharp bounds for the Stanley depth of a special class of monomial ideals of Borel type.


\end{abstract}

\section*{Introduction}

Let $K$ be a field and $S=K[x_1,\ldots,x_n]$ the polynomial ring over $K$.
Let $M$ be a \Zng. A \emph{Stanley decomposition} of $M$ is a direct sum $\mathcal D: M = \bigoplus_{i=1}^rm_i K[Z_i]$ of $\mathbb Z^n$-graded $K$-vector spaces, where $m_i\in M$ is homogeneous with respect to $\mathbb Z^n$-grading, $Z_i\subset\{x_1,\ldots,x_n\}$ such that $m_i K[Z_i] = \{um_i:\; u\in K[Z_i] \}\subset M$ is a free $K[Z_i]$-submodule of $M$. We define $\sdepth(\mathcal D)=\min_{i=1,\ldots,r} |Z_i|$ and $\sdepth(M)=\max\{\sdepth(\mathcal D)|\;\mathcal D$ is a Stanley decomposition of $M \}$. The number $\sdepth(M)$ is called the \emph{Stanley depth} of $M$. 

Stanley conjectured in \cite{stan} that $\sdepth(M)\geq \depth(M)$ for any $M$. The conjecture was disproved in \cite{duval} for $M=S/I$, where $I\subset S$ is a monomial ideal, but remains open in the case $M=I$. Herzog, Vladoiu and Zheng showed in \cite{hvz} that $\sdepth(M)$ can be computed in a finite number of steps if $M=I/J$, where $J\subset I\subset S$ are monomial ideals. In \cite{rin}, Rinaldo gave a computer implementation for this algorithm, in the computer algebra system $\mathtt{CoCoA}$ \cite{cocoa}. For an introduction in the thematic of Stanley depth, we refer the reader to \cite{herz}.

We say that a monomial ideal $I\subset S$ is of \emph{Borel type}, see \cite{her}, if it satisfies the following condition:
$(I:x_j^{\infty}) = (I:(x_1,\ldots,x_j)^{\infty}),\;(\forall)1\leq j\leq n$. The \emph{Mumford-Castelnuovo regularity} of $I$ is the number $\reg(I)=\max\{j-i\;:\; \beta_{ij}(I)\neq 0\}$, where $\beta_{ij}$'s are the graded Betti numbers. The regularity of the ideals of Borel type was extensively studied, see for instance \cite{her}, \cite{ber} and \cite{mir}. In the first section, we study the invariant $\sdepth(I)$, for an ideal of Borel type. In the general case, we note some bounds for $\sdepth(I)$, see Proposition $1.2$ and we give some tighter ones, when $I$ has a special form, see Theorem $1.6$.

\footnotetext[1]{The support from grant ID-PCE-2011-1023 of Romanian Ministry of Education, Research and Innovation is gratefully acknowledged.}

\section{Main results}

First, we recall the construction of the sequential chain associated to a Borel type ideal $I\subset S$, see \cite{her} for more details.
Assume that $\Ass(S/I)=\{P_0,\ldots,P_m\}$ with $P_i=(x_1,\ldots,x_{n_i})$, where $n\geq n_0>n_1>\cdots>n_m\geq 1$. Also, assume that
 $I=\bigcap_{i=0}^m Q_i$ is the reduced primary decomposition of $I$, with $P_i=\sqrt{Q_i}$, for all $0\leq i\leq m$.

We define $I_k:=\bigcap_{j=k}^m Q_j$, for all $0\leq k\leq m$. One can easily check that $I_i=(I_{i-1}:x_{n_{i-1}}^{\infty})$, for all $1\leq i\leq m$. The sequence of ideals $I=I_0\subset I_1\subset \cdots \subset I_m\subset I_{m+1}:=S$ is called the \emph{sequential sequence} of $I$.
Let $J_i$ be the monomial ideal generated by $G(I_i)$ in $S_i:=K[x_1,\ldots,x_{n_i}]$, for all 
$0\leq i \leq m$. Then, the saturation $J_i^{sat}=(J_i:(x_1,\ldots,x_{n_i})^{\infty}) = J_{i+1}S_i$, for all $0\leq i\leq m$, where $J_{m+1}:=S_{m+1}$. One has $I_{i+1}/I_i \cong (J_i^{sat}/J_i)[x_{n_i+1},\ldots,x_n]$. If $M=\bigoplus_{t\geq 0} M_t$ is an Artinian graded
$S$-module, we denote $s(M)=\max\{t\;:\;M_t\neq 0\}$. We recall the following result.

\begin{prop}(\cite[Corollary 2.7]{her})
$\reg(I)=\max\{s(J_0^{sat}/J_0),\ldots,s(J_m^{sat}/J_m) \}+1$.
\end{prop}

\begin{prop}
With the above notations, the following assertions hold:

$(1)$ $\sdepth(S/I_i)=\depth(S/I_i)=n-n_i$, for all $0\leq i\leq n$.

$(2)$ $\sdepth(I_0)\leq \sdepth(I_1) \leq \cdots \leq \sdepth(I_m)$.

$(3)$ $\depth(I_i)=n-n_i+1 \leq \sdepth(I_i)\leq \sdepth(P_i) =  n - \left\lfloor  \frac{n_i}{2} \right\rfloor,\;(\forall)0\leq i\leq m$.
\end{prop}

\begin{proof}
$(1)$ From \cite[Lemma 3.6]{hvz} it follows that $\sdepth(S/I_i)=\sdepth(S_{i}/J_i)+n-n_i$. Also, we have $\depth(S/I_i)=\depth(S_{i}/J_i)+n-n_i$. Since $P_iS_i = (x_1,\ldots,x_{n_i})S_i \in Ass(S_i/J_i)$, it follows that $\depth(S_i/J_i)=0$ and thus, by \cite[Theorem 1.4]{cim} or 
\cite[Proposition 18]{herz}, we get $\sdepth(S_i/J_i)=\depth(S_i/J_i)=0$.

$(2)$ Since $I_i=(I_{i-1}:x_{n_{i-1}}^{\infty})$, by \cite[Proposition 1.3]{pop} (see arXiv version), we get $\sdepth(I_{i-1}) \leq \sdepth(I_i)$, for all $1\leq i\leq m$.

$(3)$ Since $I_i=J_iS$, by \cite[Lemma 3.6]{hvz}, it follows that $\sdepth(I_i)=n-n_i+\sdepth_{S_i}(J_i)\geq n-n_i+1$. Since $P_i\in \Ass(I_i)$, it follows that there exists a monomial $v\in S$, such that $P_i=(I_i:v)$. Therefore, by \cite[Proposition 1.3]{pop} (see arXiv version), it follows that $\sdepth(P_i)\geq \sdepth(I_i)$. On the other hand, $P_i$ is generated by variables. Thus, by \cite[Lemma 3.6]{hvz} and 
\cite[Theorem 1.1]{biro}, it follows that $\sdepth(P_i)=n - \left\lfloor  \frac{ht(P_i)}{2} \right\rfloor =  n - \left\lfloor  \frac{n_i}{2} \right\rfloor$.
\end{proof}

\begin{lema}
Let $r\leq n$ and $a_1,\ldots,a_r$ be some positive integers. If $Q=(x_1^{a_1},\ldots,x_r^{a_r})\subset S$, then $reg(Q)=a_1+\cdots+a_r-r+1$.
\end{lema}

\begin{proof}
Let $\bar Q = Q\cap S'\subset S'$, where $S'=K[x_1,\ldots,x_r]$. As a particular case of Proposition $1.1$, we get $\reg(Q)=\reg(\bar Q)=s(S'/\bar Q)+1 = a_1+\cdots+a_r-r+1$.
\end{proof}

We recall the following result from \cite{ber}.

\begin{prop}(\cite[Corollary 3.17]{ber})
If $I\subset S$ is an ideal of Borel type with the irredundant irreducible decomposition $I=\bigcap_{i=1}^r C_i$, then 
$\reg(I)=\max\{ \reg(C_i):\; 1\leq i\leq r\}$.
\end{prop}

Let $n\geq n_0 > n_1 > \cdots > n_m \geq 1$ be some integers. 
Let $a_{ij}$ be some positive integers, where $0\leq i\leq m$ and $1\leq j\leq n_i$.
We consider the monomial irreducible ideals $Q_i=(x_1^{a_{i1}},\ldots,x_{n_i}^{a_{in_i}})$, for $0\leq i\leq m$. Let $I_i:=\bigcap_{j=i}^m Q_j$ and denote $I=I_0$. Since $P_i=(x_1,\ldots,x_{n_i}) = \sqrt{Q_i}$ for all $0\leq i\leq m$, by \cite[Proposition 5.2]{hp} or \cite[Corollary 1.2]{mir}, it follows that $I$ is an ideal of Borel type. As a direct consequence of Lemma $1.3$ and Proposition $1.4$, we get the following corollary.

\begin{cor}
If $a_{ij}\geq a_{i+1,j}$ for all $j\leq n_{i+1}$ and $i<m$, then
$\reg(I_i) = \reg(Q_i) = a_{i1}+a_{i2}+\cdots+a_{in_i}-n_i+1$, for all $0\leq i\leq m$.
\end{cor}


\begin{teor}
If $a_{ij}\geq a_{i+1,j}$ for all $j\leq n_{i+1}$ and $i<m$, then for all $0\leq i\leq m$, it holds that
$$n - \left\lfloor  \frac{n_i}{2} \right\rfloor \geq \sdepth(I_i) \geq n + \left\lceil \frac{n_m}{2} \right\rceil - n_i.$$ 
\end{teor}

\begin{proof}
The first inequality follows from Proposition $1.2(3)$. In order to prove the second one, we use induction on $i\leq m$. If $i=m$, then $I_m=Q_m$ is an irreducible ideal, and therefore, by \cite[Theorem 1.3]{mirc}, $\sdepth(I_m)=n-\left\lfloor \frac{n_m}{2} \right\rfloor =  n + \left\lceil \frac{n_m}{2} \right\rceil - n_m$.

Assume $i<m$. We can write $Q_i=U_i + V_i$, where $U_i=(x_1^{a_{i1}},\ldots,x_{n_{i+1}}^{a_{in_{i+1}}})$ and $V_i=(x_{n_{i+1}+1}^{a_{i,n_{i+1}+1}},\ldots,x_{n_i}^{a_{in_i}})$. Since $a_{ij}\geq a_{i+1,j}$ for all $j\leq n_{i+1}$, it follows that $U_i\subset Q_{i+1}$. Therefore, $I_i = (U_i+V_i)\cap I_{i+1} = (U_i\cap I_{i+1})+(V_i\cap I_{i+1})$. Note that $J:=U_i\cap I_{i+1} = U_i\cap I_{i+2}$ is a Borel type ideal with the irreducible irredundant decomposition $J=U_i\cap Q_{i+2}\cap \cdots \cap Q_m$, and, therefore, of the same class as $I_{i+1}$. Thus, by induction hypothesis, it follows that $\sdepth(J)\geq n + \left\lceil \frac{n_m}{2} \right\rceil - n_{i+1}$.

On the other hand, by \cite[Remark 1.3]{sir} and the induction hypothesis, 
$\sdepth(V_i\cap I_{i+1}) \geq \sdepth(V_i)+\sdepth(I_{i+1}) - n = \sdepth(I_{i+1}) - \left\lfloor \frac{n_i-n_{i+1}}{2} \right\rfloor$.

Let $\bar V_i\subset S'=K[x_{n_{i+1}+1},\ldots,x_{n_i}]$ be the monomial ideal generated by $G(V_i)$ and let $\bar J \subset S''=K[x_1,\ldots,x_{n_{i+1}},x_{n_i+1},\ldots,x_n]$ be the monomial ideal generated by $G(J)$. Since $J\subset I_{i+1}$, it follows that $I_i = (\bar J \otimes_K (S''/\bar V_i)) \oplus (V_i\cap I_i)$. By \cite[Proposition 2.10]{bruns} and \cite[Lemma 2.2]{asia}, we get:
$$\sdepth(I_i)\geq \min\{ \sdepth(J)-n_i+n_{i+1} , \sdepth(I_{i+1}) - \left\lfloor \frac{n_i-n_{i+1}}{2} \right\rfloor \} \geq n + \left\lceil \frac{n_m}{2} \right\rceil - n_{i},$$ as required.
\end{proof}

\noindent
\textbf{Question}: What can we say about the case when the condition $a_{ij}\geq a_{i+1,j}$ is removed? Of course, the method used in the proof of the Theorem $1.6$ do not work. However, our computer experiments in $\mathtt{Cocoa}$ \cite{cocoa} suggested that the conclusion of the Theorem $1.6$ might be true. Unfortunately, we are not able to give either a proof, or a counterexample.

The next example shows that the bounds given in Theorem $1.6$ are sharp.

\begin{exm}
\emph{Let $I=Q_0\cap Q_1$, where  $Q_0=(x_1^3,x_2^2,x_3^2,x_4,x_5)$ and $Q_1= (x_1,x_2,x_3,x_4)$ are ideals in $S=K[x_1,\ldots,x_5]$. Then $I_1=Q_1$ and $\sdepth(I_1)=5-\left\lfloor \frac{4}{2} \right\rfloor = 3$. Also $n=5$, $n_0=5$ and $n_1=4$. Using $\mathtt{CoCoa}$, we get $\sdepth(I)=2=n-\left\lceil \frac{n_1}{2} \right\rceil - n_0$. Let $Q'_0=(x_1^2,x_2^2,x_3,x_4,x_5)\subset S$ and $I'=Q'_0\cap Q_1$. Using $\mathtt{CoCoa}$ \cite{cocoa}, we get $\sdepth(I')=3 = n - \left\lfloor  \frac{n_0}{2} \right\rfloor$.}
\end{exm}

\vspace{2mm} \noindent {\footnotesize
\begin{minipage}[b]{15cm}
Mircea Cimpoea\c s, Simion Stoilow Institute of Mathematics, Research unit 5, P.O.Box 1-764,\\
Bucharest 014700, Romania\\
E-mail: mircea.cimpoeas@imar.ro
\end{minipage}}
\end{document}